\documentclass[reqno, 11pt, oneside]{amsart}
\usepackage[a4paper, total={6in, 8in}]{geometry}
\usepackage{xcolor}
\usepackage{amsthm,amsmath,amssymb}
\usepackage{hyperref}
\hypersetup{
    colorlinks=true,
    linkcolor=blue,
    filecolor=magenta,      
    urlcolor=cyan,
    citecolor=teal,
    }
\theoremstyle{plain}
\newtheorem{theorem}{Theorem}
\newtheorem{lemma}[theorem]{Lemma}
\newtheorem{corollary}[theorem]{Corollary}
\newenvironment{thmbis}[1]
 {\renewcommand{\thetheorem}{\ref{#1}}%
 \addtocounter{theorem}{-1}%
 \begin{theorem}}
 {\end{theorem}}

\theoremstyle{remark}
\newtheorem{remark}[theorem]{Remark}
\newtheorem{definition}[theorem]{Definition}
\newtheorem{example}[theorem]{Example}

\numberwithin{equation}{section}

\def\Gr{\operatorname{Gr}}
\def\dimB{\operatorname{dim_{B}}}
\def\dimH{\operatorname{dim_{H}}}

\allowdisplaybreaks

\title[Hausdorff dimension of Graphs of Limit Functions]{Hausdorff dimension of Graphs of Limit Functions Generated by Quasi-Linear Functions}
\author[W. Wu]{Wen Wu}
\address[W. Wu]{School of Mathematics, South China University of Technology, Guangzhou 510640, People's Republic of China}
\email[corresponding author]{wuwen@scut.edu.cn}

\author[S. Zhong]{Sheng Zhong}
\address[S. Zhong]{School of Mathematics, South China University of Technology, Guangzhou 510640, People's Republic of China}
\curraddr{School of Mathematical Sciences, Beihang University, Beijing 102206, People's Republic of China}
\email{2381923039@qq.com}

\thanks{Wen Wu is the corresponding author.}
\subjclass[2020]{Primary 28A80, Secondary 11B85}

\keywords{Graphs of functions, Hausdorff dimension, Rudin-Shapiro sequence}
\date{}

\begin{document}
\linespread{1.05}\selectfont

\begin{abstract}
    The limit functions generated by quasi-linear functions or sequences (including the sum of the Rudin-Shapiro sequence as an example) are continuous but almost everywhere non-differentiable functions. Their graphs are fractal curves. In 2017 and 2020, Chen, L\"u, Wen and the first author studied the box dimension of the graphs of the limit functions.

    In this paper, we focus on the Hausdorff dimension of the graphs of such limit functions. We first prove that the Hausdorff dimension of the graph of the limit function generated by the abelian complexity of the Rudin-Shapiro sequence is $\frac{3}{2}$. Then we extend the result to the graphs of limit functions generated by quasi-linear functions.
\end{abstract}

\maketitle

\section{Introduction}

The graphs of continuous but nowhere differentiable functions, as fractal sets, have been studied for many years since the work of Besicovitch and Ursell \cite{BU1937}. A well-known example of such functions is the Weierstrass function:
\[
W_{a,b}(x) = \sum_{n=1}^{\infty} a^n \cos(b^n x)
\]
where \(0 < a < 1\), \(b > 1\), and \(ab > 1\). Based on the Weierstrass function, a class of continuous and nowhere differentiable functions has been constructed—the Weierstrass-type function:
\[
W_{a,b,\theta}^{\phi}(x) = \sum_{n=1}^{\infty} a^n \phi(b^n x + \theta)
\]
where \(0 < a < 1\), \(b > 1\), \(ab > 1\), \(\theta \in \mathbb{R}\), and \(\phi\) is an integral-periodic Lipschitz function. In particular, when \(\phi(x) = \operatorname{dist}(x,\mathbb{N})\) (the distance between \(x\) and the set of natural numbers \(\mathbb{N}\)), \(W_{a,b,\theta}^{\phi}(x)\) is the Takagi function. The fractal properties of the graphs of Weierstrass-type functions—including the Takagi function—such as determining various fractal dimensions of the graphs and their level sets, have been extensively studied (see, for example, \cite{BBR2014,R1988,MW1986,S2018,F2020} and the references therein). It is an interesting question to consider: what if $\phi$ is a non-periodic function?
 
Brillhart et al. \cite{BM1978, BEM1983} investigated the limiting behavior of the sum of the Rudin-Shapiro sequence $\mathbf{r}$ and introduced a continuous limit function that is almost everywhere non-differentiable. In 2017, L\"u et al. \cite{LCWW17} further extended this line of research: they analyzed the limit function generated by the abelian complexity of $\mathbf{r}$ and determined its box dimension. Specifically, let $\rho(n)$ be the abelian complexity function of $\mathbf{r}$; in \cite{LCWW17}, the authors showed that the limit function $\lambda$ of $\rho(\cdot)$, defined by  
\[
\lambda(x) = \lim_{k \to \infty} \frac{\rho(4^k x)}{\sqrt{4^k x}},
\]  
is continuous and almost everywhere non-differentiable. Here, $\rho(x)$ is extended to positive real numbers by $\rho(x) := \rho(\lfloor x \rfloor)$, where $\lfloor x \rfloor$ denotes the floor function (integer part) of $x$. Additionally, they proved that the box dimension of the graph of $\lambda$ on the interval $(0, 1)$ is $3/2$.  

We note that a variation of $\lambda(x)$ admits a series representation, which implies $\lambda$ is not a Weierstrass-type function. Specifically, for all $x > 0$,  
\[
\sqrt{x}\lambda(x) - \rho(x) = \sum_{j=1}^{\infty} \phi(x) 2^{-j},
\]  
where $\phi(x)$ is non-periodic; see Lemma \ref{lem:series} or \cite[Proposition 2]{LCWW17} for further details.  

Subsequently, L\"u et al. \cite{LCWW20} generalized these results to a class of unbounded sequences (including certain $k$-regular sequences; for the definition of $k$-regular sequences, see \cite[Chapter 16]{AS03}) referred to as quasi-linear functions (or quasi-linear sequences). Notably, while the box dimension of the graphs of these limit functions has been established in prior works, their Hausdorff dimension remains undetermined.

In this paper, we determine the Hausdorff dimension of the limit function $\lambda$ generated by the abelian complexity function of $\mathbf{r}$ and a class of quasi-linear functions. Our first result says that the Hausdorff dimension of the graph of $\lambda(x)$ on $(0,1)$ is $3/2$.
\begin{theorem}\label{mainthm:1}
    Let $\lambda(x)$ be defined as in \eqref{def-eq:lambda}. For all $u,v\in (0,1]$ with $u<v$, 
    \[\dimH\{(x,\lambda(x))\mid u<x<v\}=\frac{3}{2}\]
    where $\dimH$ stands for the Hausdorff dimension.
\end{theorem}

Our next result extend Theorem \ref{mainthm:1} to the limit functions generated by quasi-linear functions. Let $b\geq 2$ be an integer and let $\mathbf{s}=(s(n))_{n\geq 0}$ be an integer sequence. Set 
    \begin{align*}
        \alpha & = \alpha(\mathbf{s}):=\inf\Bigl\{x\geq 0\mid \limsup\limits_{n\to+\infty}\frac{|s(n)|}{n^{x}}=0 \Bigr\},\\
        \beta & = \beta(\mathbf{s}):=\inf\Bigl\{x\geq 0\mid \limsup\limits_{n\to+\infty}\frac{|s(n+1)-s(n)|}{n^{x}}=0 \Bigr\}.
    \end{align*} 
\begin{definition}[Quasi-linear functions]\label{def:ql}
    Let $\mathbf{s}:\mathbb{N}\to\mathbb{Z}$. If $\alpha(\mathbf{s})>\beta(\mathbf{s})$ and there is a constant $C>0$ such that for all positive integers $n$ and $0\leq i\leq b-1$,
\begin{equation}\label{eq:ql-2}
    |s(bn+i)-b^{\alpha}s(n)| \leq C n^{\beta},
\end{equation}
then the sequence $\mathbf{s}$ is called \emph{quasi-linear} (in base $b$). 
\end{definition}

We extend the domain of $\mathbf{s}$ to positive real numbers by assigning $s(x):=s(\lfloor x\rfloor)$ for all $x\geq 0$. The limit function of a quasi-linear sequence $s$ is \[\lambda_{\mathbf{s}}(x):=\lim_{k\to\infty}\frac{s(b^k x)}{(b^k x)^\alpha},\quad x> 0.\] 
Let \[a_\mathbf{s}(x):=x^{\alpha}\lambda_\mathbf{s}(x)-s(x).\] From the series representation given in Lemma \ref{thm:cite-quasi-linear} (1), we can see that $a_{\mathbf{s}}$ is also not a Weierstrass-type function.
% It was shown in \cite{LCWW20} that \[x^{\alpha}\lambda_{\mathbf{s}}(x)-s(x)=\sum_{j=1}^{\infty}c(j,x)\,b^{-\alpha j}=:a_{\mathbf{s}}(x)\] where $c(j,x):=s(b^jx)-b^{\alpha}s(b^{j-1}x)$ is non-periodic. So $\lambda_{\mathbf{s}}$ is also not a Weierstrass type function.

\begin{theorem}\label{mainthm:2}
    Let $\mathbf{s}$ be a quasi-linear sequence with $\alpha(\mathbf{s})=\alpha$ and $\beta(\mathbf{s})=0$. Assume that there is an (increasing) sequence $(t_n)_{n\geq 1}$ and a constant $c>0$ such that  
    \begin{equation*}
        a_{\mathbf{s}}(t_{n+1}b^{-k})-a_{\mathbf{s}}(t_{n}b^{-k})>c\,b^{-\alpha k}
    \end{equation*}
    for all $k\geq 1$ and for all $1\leq t_n < t_{n+1}\leq b^k-1$. Then for all $u,v\in (0,1]$ with $u<v$,  \[\dimH\{(x,\lambda_{\mathbf{s}}(x))\mid u<x<v\} = 2-\alpha.\]
\end{theorem}

The paper is organized as follows. In Section 2, we first discuss the relationship between the Hausdorff dimensions of the graphs of $a(x)$ and $\lambda(x)$, and then establish a lower bound for the Hausdorff dimension of $\lambda(x)$’s graph. Since this lower bound coincides with the known box dimension, we complete the proof of Theorem \ref{mainthm:1}. In Section 3, following a similar approach, we show that the upper bound for the upper box dimension derived in \cite{LCWW20} also serves as a lower bound for the Hausdorff dimension of $\lambda_{\mathbf{s}}(x)$’s graph, thereby proving Theorem \ref{mainthm:2}.

\section{Hausdorff dimension of the graph of \texorpdfstring{$\lambda(x)$}{lambda(x)}}

In this section, we show that the Hausdorff dimension of the graph of $\lambda(x)$ on any interval is $3/2$. Let $\mathbf{r} = (r_0, r_1, r_2, \dots)$ denote the Rudin-Shapiro sequence on $\{-1, 1\}$, defined by $r_0 = 1$ and
\begin{equation}\label{def:rs}
    r_{2n} = r_n,\quad r_{2n+1} = (-1)^n r_n \quad (n \ge 0).
\end{equation}
Let $\rho(n)$ be the abelian complexity function of $\mathbf{r}$. As shown in \cite[Theorem 1]{LCWW17}, the sequence $(\rho(n))_{n \ge 1}$ is $2$-regular and satisfies the recurrence relation: $\rho(1) = 2$, $\rho(2) = 3$, $\rho(3) = 4$, and for all $n \ge 1$,
\begin{align*}
    \rho(4n) & = 2\rho(n) + 1, & \rho(4n+1) & = 2\rho(n),\\
    \rho(4n+2) & = \rho(n) + \rho(n+1), & \rho(4n+3) & = 2\rho(n+1).
\end{align*}
Recall that the limit function $\lambda(\cdot)$ is defined as
\begin{equation}\label{def-eq:lambda}
    \lambda(x) = \lim_{k \to \infty} \frac{\rho(4^k x)}{\sqrt{4^k x}}.
\end{equation}
The value of $\lambda(x)$ is related to the $4$-ary expansion of $x$. For all $x > 0$, let $x = \sum_{j=0}^{\infty} x_j 4^{-j}$ be its $4$-ary expansion, where $x_0 \in \mathbb{N}$ and $x_i \in \{0, 1, 2, 3\}$. Assuming there are infinitely many indices $j$ such that $x_j \neq 3$, every such $x$ has a unique $4$-ary expansion. 

The following result provides a series expression for $a(x) := \sqrt{x}\lambda(x) - \rho(x)$.
\begin{lemma}[Proposition 2 in \cite{LCWW17}] \label{lem:series}
For all $x>0$ with $x=\sum_{j=0}^{\infty}x_j 4^{-j}$, we have $$a(x):=\sqrt{x}\lambda(x)-\rho(x) = \sum_{j=1}^{\infty}d(x_j)a_j(x)2^{-j}$$ where for $j=1,2,\dots$, \[a_j(x)=\begin{cases}
    -1, & \text{ if }4^{j}x<1,\\
    \Delta\rho(\lfloor 4^{j}x\rfloor-1), & \text{ otherwise}
\end{cases}\] and $d(y):=|x-1|$ for $y=0,1,2,3$. 
\end{lemma}
\begin{remark}
    It follows from the recurrence relations of $\rho$ that $|\Delta\rho(n)| = 1$ for all $n\geq 1$. Thus $|a_j(x)|\equiv 1$.
\end{remark}
Let $f$ be a real-valued function on $\mathbb{R}$. For all $u,v\in\mathbb{R}^+$ with $u<v$, the graph of the function $f$ over the interval $(u,v)$ is denoted by \[\mathrm{Gr}_{u,v}(f):=\bigl\{(x,f(x))\mid  u< x< v\bigr\}.\] The box dimension, denoted as $\dimB$, of the graph of $f$ is obtained in \cite{LCWW17}.
\begin{lemma}[cf. \cite{LCWW17}, Theorem 3]\label{lem:dimb}
    For all $u,v\in(0,1)$ with $u<v$, $\dimB\Gr_{u,v}(\lambda)=3/2$.
\end{lemma}
According to Lemma \ref{lem:dimb}, we only need to show that $\dimH\Gr_{u,v}(\lambda)\geq 3/2$.

\begin{lemma}\label{lem:a-lambda}
    For all $u,v\in (0,1)$ with $u<v$, we have $\dimH\Gr_{u,v}(a)\leq \dimH\Gr_{u,v}(\lambda)$.
\end{lemma}
\begin{proof}
    Let $\phi:\mathbb{R}^2\to\mathbb{R}^2$ be the mapping defined as $(x,a(x))\mapsto (x,\lambda(x))$ for all $x\in(u,v)$. Note that $\Gr_{u,v}(\lambda)=\phi(\Gr_{u,v}(a))$. We only need to show that $\phi$ is a Lipschitz mapping.

    Note that the mapping $x\mapsto \sqrt{x}$ satisfies the Lipschitz condition on $(u,v)$ and $\lambda(x)$ is $\tfrac{1}{2}$-H\"older (see the proof of \cite[Theorem 3]{LCWW17}). Then there exist positive constants $C_1$ and $C_2$ $(<+\infty)$ such that for all $x_1,x_2\in (u,v)$,  \[|\sqrt{x_1}-\sqrt{x_2}|\leq C_1 |x_1-x_2| \quad \text{and} \quad |\lambda(x_2)|\leq C_2.\]
    For all $x_1,x_2\in (u,v)$, since $\rho(x_1)=\rho(x_2)=\rho(0)$, we have 
    \begin{align*}
        & \sqrt{(x_1-x_2)^2+(a(x_1)-a(x_2))^2} = \sqrt{(x_1-x_2)^2+(\sqrt{x_1}\lambda(x_1)-\sqrt{x_2}\lambda(x_2))^2}\\
        =\ &  \sqrt{(x_1-x_2)^2+\bigl(\sqrt{x_1}\lambda(x_1)-\sqrt{x_1}\lambda(x_2)+\sqrt{x_1}\lambda(x_2)-\sqrt{x_2}\lambda(x_2)\bigr)^2}\\
        \leq \ & \sqrt{(x_1-x_2)^2+2x_1(\lambda(x_1)-\lambda(x_2))^2+2\lambda(x_2)^2(\sqrt{x_1}-\sqrt{x_2})^2}\\
        \leq \ & \sqrt{(x_1-x_2)^2+2x_1(\lambda(x_1)-\lambda(x_2))^2+2\lambda(x_2)^2(\sqrt{x_1}-\sqrt{x_2})^2}\\
        \leq \ & \sqrt{(x_1-x_2)^2+2v(\lambda(x_1)-\lambda(x_2))^2+2C_1^2C_2^2(x_1-x_2)^2}\\
        \leq \ & C_3\sqrt{(x_1-x_2)^2+(\lambda(x_1)-\lambda(x_2))^2}
    \end{align*}
    where $C_3=\max\{\sqrt{2v},\sqrt{1+2C_1^2C_2^2}\}$. This shows that $\phi$ is a Lipschitz mapping. 
\end{proof}

To establish the lower bound of the Hausdorff dimension of the graph of \(a(x)\), we utilize the mass distribution principle (see, for example, \cite[Lemma 1.2.8]{BP2017} or \cite[Section 4.1]{F2004}).
\begin{lemma}[Mass Distribution Principle]\label{lem:mdp}
Let \(F \subset \mathbb{R}^n\). If \(F\) supports a strictly positive Borel measure \(\mu\) such that \(\mu(B(x, r)) \leq Cr^t\) for some constant \(0 < C < \infty\) and all balls \(B(x, r)\) (with center \(x\) and radius \(r\)), then \(\dimH F \geq t\).
\end{lemma}
Thus, we need a strictly positive Borel measure supported on the graph of \(a(x)\). For all $n\geq 1$, let \[f_n(x):=\sum_{j=1}^{n}d(x_j)a_j(x)2^{-j}.\] Note that for all $x\in [0,1]$, 
\begin{align}
    |a(x)-f_n(x)| & = \Bigl|\sum_{j=1}^{\infty}a_j(x)d(x_j)2^{-j}-\sum_{j=1}^{n}a_j(x)d(x_j)2^{-j}\Bigr|\notag\\ 
    & = \Bigl|\sum_{j=n+1}^{\infty}a_j(x)d(x_j)2^{-j}\Bigr|\notag\\
    & \leq 2 \sum_{j=n+1}^{\infty}2^{-j}=2^{-n+1}.\label{eq:a-fn}
\end{align} 
Then $\{f_n(x)\}_{n\geq 1}$ converges to $a(x)$ uniformly on $[0,1]$. Let $I_{n,k}$ be the $4$-adic interval $[\frac{k}{4^n},\frac{k+1}{4^n})$ where $0\leq k<4^n$. Denote by $I_{n}(x):=[\tfrac{i}{4^n},\tfrac{i+1}{4^n})$ ($i\in\mathbb{N}$) the $4$-adic interval that contains $x$.
\begin{lemma}\label{lem:f-step}
    Let $x\in [0,1)$ with $x=\sum_{j=1}^{\infty}x_j4^{-j}$. 
    \begin{enumerate}
        \item For every $n\geq 1$, $f_n(x)$ is a step function and its graph is composed of $4^n$ horizontal line segments of length $4^{-n}$. More precisely, for all $k=0,1,\dots,4^n-1$, $f_n(x)=f_n(\tfrac{k}{4^n})$ for all $x\in I_{n,k}$.
        \item If $x_n=1$ for some $n\geq 1$, then for all $y\in I_{n-1}(x)$, $|f_n(x)-f_n(y)|\leq 2^{-n+1}$.
    \end{enumerate}
\end{lemma}
\begin{proof}
    (1) Since for all $x\in I_{n,k}$, $\lfloor 4^jx\rfloor=\lfloor 4^{j-n} k \rfloor$. According to the definition of $a_j(x)$, we have  
    \begin{equation}\label{eq:aj}
        a_j(x)=a_j(k)=\begin{cases}
        -1, & \text{ if }4^{j-n}k < 1,\\
        \Delta\rho(\lfloor 4^{j-n}k\rfloor-1), & \text{ otherwise},
    \end{cases}
    \end{equation}
    for all $j=0,1,\dots, n$. This shows that $a_j(x)$ ($j=0,1,\dots,n$) are constant on $I_{n,k}$. 

    Note that for all $x\in I_{n,k}$, $x_j=\lfloor 4(4^{j-1}k\mod 1)\rfloor=:k_j$ and $d(x_j)=d(k_j)$ for $j=1,\dots,n$. Consequently, for all $x\in I_{n,k}$, \[f_n(x)=\sum_{j=1}^{n}d(x_j)a_j(x)2^{-j}=\sum_{j=1}^{n}d(k_j)a_j(k)2^{-j}=f_n(k/4^n).\] So $f_n(x)$ takes a constant value on $I_{n,k}$.

    (2) Write $y=\sum_{j=1}^{\infty}y_j4^{-j}$. For all $y\in I_{n-1}(x)$, we have  \[y_j=x_j \quad\text{and}\quad d(x_j)=d(y_j)\] for $j=1,\dots,n-1$. Moreover, it follows from Equation \eqref{eq:aj} that $a_j(x)=a_j(y)$ for $j=0,1,\dots,n-1$. Note also that $d(x_n)=d(1)=0$. We have
    \begin{align*}
        |f_n(x)-f_n(y)| & = \left|\sum_{j=1}^n a_j(x)d(x_j)2^{-j}-\sum_{j=1}^n a_j(y)d(y_j)2^{-j}\right|\\
        & = |a_n(x)d(x_n)2^{-j}-a_n(y)d(y_n)2^{-j}|\\
        & = |a_n(y)d(y_n)2^{-j}|\\
        & \leq 2^{-n+1}
    \end{align*}
    where the inequality holds since $|a_n(x)|=1$ and $d(y_n)\leq 2$.
\end{proof}

By Lemma \ref{lem:f-step} (1), $f$ is a step function. For $k=0,1,\dots,4^n-1$, let \[F_{n,k}=[\tfrac{k}{4^n},\tfrac{k+1}{4^n})\times [f_n(\tfrac{k}{4^n})-2^{-n+1},f_n(\tfrac{k}{4^n})+2^{-n+1}]\] be a rectangle which covers the image $f(I_{n,k})$. Let  
\begin{equation}\label{lem:def-Fn}
    \mathcal{F}_n=\{F_{n,k}\}_{k=0}^{4^n-1}\quad\text{and}\quad 
    F_n=\bigcup_{k=0}^{4^n-1}F_{n,k}.
\end{equation} We see that $\Gr(f_n)\subset F_n$. Next, we show that $\{\mathcal{F}_n\}$ has a nested structure. 

\begin{lemma}\label{lem:set-Fn}
    Let $\{\mathcal{F}_n\}_{n\geq 1}$ be families of sets given in Equation \eqref{lem:def-Fn}. Then for all $n\geq 1$, \[F_{n+1}\subset F_n \quad \text{and} \quad F_{n+1,4k+i}\subset F_{n,k}\] for $k=0,1,\dots,4^n-1$ and $i=0,1,2,3$.
\end{lemma}
\begin{proof}
    Note that \(I_{n,k} = \bigcup_{i=0}^{3} I_{n+1,4k+i}\). Fix \(x \in I_{n+1,4k+1}\). We have \(\lfloor 4^{n+1}x \rfloor \equiv 1 \pmod{4}\) and \(d(x_{n+1}) = 0\). Therefore,
    \begin{equation}\label{eq:n_to_n+1}
        f_{n+1}(x) = \sum_{j=1}^{n+1} a_j(x)d(x_j) = \sum_{j=1}^{n} a_j(x)d(x_j) = f_n(x).
    \end{equation}
    For all \(y \in I_{n+1,4k+i} \subset I_{n,k}\) (where \(i = 0,1,2,3\)), by Equation \eqref{eq:n_to_n+1} and Lemma \ref{lem:f-step} (2), we obtain
    \[
    |f_{n}(x) - f_{n+1}(y)| = |f_{n+1}(x) - f_{n+1}(y)| \leq 2^{-n},
    \]
    which implies
    \[
    \left[ f_{n+1}(y) - 2^{-n}, f_{n+1}(y) + 2^{-n} \right] \subset \left[ f_n(x) - 2^{-n+1}, f_n(x) + 2^{-n+1} \right].
    \]
    Consequently, \(F_{n+1,4k+i} \subset F_{n,k}\) for each \(i = 0,1,2,3\), and
    \[
    F_{n+1} = \bigcup_{k=0}^{4^n - 1} \bigcup_{i=0}^{3} F_{n+1,4k+i} \subset \bigcup_{k=0}^{4^n - 1} F_{n,k} = F_n,
    \]
    which completes the proof.
\end{proof}

Now we are ready to proof Theorem \ref{mainthm:1}.
\begin{thmbis}{mainthm:1}
    For all $u,v\in (0,1]$ with $u<v$, $\dimH\Gr_{u,v}(\lambda)=3/2$.
\end{thmbis}
\begin{proof}
    It is known in \cite[Theorem 3]{LCWW17} that $\dimB\Gr_{u,v}(\lambda)=3/2$. By Lemma \ref{lem:a-lambda}, to show that $\dimH\Gr_{u,v}(\lambda)=3/2$,  we only need to show that $\dimH\Gr_{u,v}(a)\geq 3/2$. For $u,v\in(0,1]$ with $u<v$, there is $n_0$ and $k_0$ such that $I_{n_0,k_0}=[\frac{k_0}{4^{n_0}},\frac{k_0+1}{4^{n_0}})\subset (u,v)$. Let $u_0=\frac{k_0}{4^{n_0}}$ and $v_0=\frac{k_0+1}{4^{n_0}}$. If we show that $\dimH \Gr_{u_0,v_0}(a)\geq 3/2$, then $\dimH\Gr_{u,v}(a)\geq  \dimH \Gr_{u_0,v_0}(a)\geq 3/2$. In the following, we first construct a positive Borel measure supported on the graph of $a(x)$. Then apply the mass distribution principle to obtain the lower bound of the Hausdorff dimension of $\Gr_{u,v}(a)$.

    Note that $\mathcal{F}_{n}=\{F_{n,0}, F_{n,1}, \dots, F_{n,4^n-1}\}$ is a collection of disjoint sets. We define a mass distribution $\mu$ on $F_1$ by assigning $\mu(F_{n,k})=4^{-n}$ for all $n\geq 1$ and $k=0,1,\dots,4^{n}-1$. Using the nested structure showed in Lemma \ref{lem:set-Fn}, we have \[1=\mu(F_n)=\sum_{k=0}^{4^n-1}\mu(F_{n,k}) \text{ and } \mu(F_{n,k})=\sum_{i=0}^{3}\mu(F_{n+1,4k+i}).\] For each $n\geq 1$, we define $\mu(A)=0$ for all Borel sets $A$ with $A\cap F_n=\varnothing$. Then $\mu$ can be extended to a Borel measure on $\mathbb{R}^2$ and the support of $\mu$ is contained in $\cap_{n=1}^{\infty}\bar{F}_n$. By Equation \eqref{eq:a-fn} and the definition of $F_{n_0,k_0}$, we have $\Gr_{u_0,v_0}(a)\subset F_{n_0,k_0}$. Let $\tilde{\mu}$ be the restriction of $\mu$ on $F_{n_0,k_0}$. Namely, $\tilde{\mu}(A):=\mu(A\cap F_{n_0,k_0})/\mu(F_{n_0,k_0})=4^{n_0}\mu(A\cap F_{n_0,k_0})$ for all Borel sets $A$.

    For any Borel set $U$ with $|U|<1$, we have $4^{-n-1}\leq |U| < 4^{-n}$ for some integer $n\geq n_0$. Then $U$ is covered by a square, denoted by $S$, of side length $4^{-n}$. Lemma 8 in \cite{LCWW17} says for all $k\geq 1$ and $z=1,2,\dots,4^k-1$, 
    \begin{equation*}%\label{eq:diff-a}
        a((z+1)4^{-k})-a(z\cdot 4^{-k}) = \begin{cases}
            2^{-k}, & \text{if }z\leq 4^k-2,\\
            2^{-k}-1, & \text{if }z=4^k-1.
        \end{cases}
    \end{equation*}
    For all $k>2n$, the monotonicity of $a(\cdot)$ on the discrete points $\{z\cdot 4^{-k}\}_z$ implies that $S$ can contain at most $\tfrac{4^{-n}}{2^{-k}}+1$ points of the form $(4^{-k}z, a(4^{-k}z))$. By Equation \eqref{eq:a-fn} and the definition of $F_{k,z}$, we see that $(4^{-k}z, a(4^{-k}z))\in F_{k,z}$. Then $S$ intersects at most $\tfrac{4^{-n}}{2^{-k}}+3$ sets $F_{k,z}$ for $z=1,2,\dots, 4^k-1$. Therefore, 
    \begin{align*}
        \tilde{\mu}(U) & \leq  \tilde{\mu}(S) \leq \bigcup_{z:\, F_{k,z}\cap S\neq\varnothing}\tilde{\mu}(F_{k,z}) \leq (\tfrac{4^{-n}}{2^{-k}}+3)\cdot 4^{n_0-k}\\
        & \leq  4^{n_0}(4^{-n}\cdot 2^{-2n}+3\cdot 4^{-2n})\leq 4^{n_0+1}\cdot  2^{-3n}\leq 4^{n_0+3}\, |U|^{3/2}.
    \end{align*}
    Then it follows from Lemma \ref{lem:mdp} (the Mass Distribution Principle), $\dimH\Gr_{u_0,v_0}(a)\geq 3/2$.
\end{proof}

\begin{corollary}
    $\dimH\Gr_{0,1}(\lambda)=3/2$.
\end{corollary}
\begin{proof}
    Since removing countable points from $\Gr_{0,1}(\lambda)$ does not change its Hausdorff dimension, we have $\dimH\Gr_{0,1}(\lambda) = \dimH\bigcup_{j=1}^{\infty}\Gr_{\frac{1}{j+1},\frac{1}{j}}(\lambda)$. By Theorem \ref{mainthm:1}, $\dimH\Gr_{\frac{1}{j+1},\frac{1}{j}}(\lambda)=3/2$ for all $j\geq 1$. Then $\dimH\Gr_{0,1}(\lambda) =3/2$.
\end{proof}

\begin{remark}
    The graph of a Brownian motion is a continuous but nowhere differentiable curve almost surely, and its fractal properties have been studied in, for example, \cite{A1978,HY2019, T1955}. While the graph of a one-dimensional Brownian motion $B(t)$ has a Hausdorff dimension of $3/2$ almost surely, one might wonder: is $\lambda(x)$ a realization of a Brownian motion? In fact, it is not. If $\lambda(x)$ were a Brownian motion, then for all $t, s > 0$ with $t > s$, the increment $\lambda(t) - \lambda(s)$ would follow a zero-mean Gaussian distribution $N(0, t-s)$. However, it is easy to see from the definition of $\lambda(x)$ that $\lambda(4x) = \lambda(x)$. This implies $\lambda(4x) - \lambda(x) = 0$, which clearly does not follow the Gaussian distribution $N(0, 4x-x)$. Therefore, $\lambda(x)$ cannot be a Brownian motion almost surely.
\end{remark}

\section{Graphs of limit functions generated by quasi-linear sequences}
In this section, we extend Theorem \ref{mainthm:1} to the limit functions of quasi-linear sequences. Under certain conditions, we obtain the Hausdorff dimension of the graphs of the limit functions.

Let \( b \geq 2 \) be an integer and let \( \mathbf{s} = (s(n))_{n \geq 0} \) be an integer sequence. Suppose that \( \mathbf{s} \) is a quasi-linear sequence as given in Definition \ref{def:ql}. Recall that the limit function of \( \mathbf{s} \) is defined by
\begin{equation}\label{def:lambda-s}
    \lambda_{\mathbf{s}}(x) := \lim_{k \to \infty} \frac{s(b^k x)}{(b^k x)^\alpha}, \quad x > 0.
\end{equation}
In \cite{LCWW20}, L\"u et al. provided a precise series representation for the function
\[ a_{\mathbf{s}}(x) := x^\alpha \lambda_{\mathbf{s}}(x) - s(x) \]
and derived an upper bound for the upper box dimension of the graph of \( \lambda_{\mathbf{s}}(x) \).

\begin{lemma}[cf. \cite{LCWW20}, Prop. 10 \& Cor. 18]\label{thm:cite-quasi-linear}
    Let \( \mathbf{s} \) be a quasi-linear sequence and let \( \lambda_{\mathbf{s}} \) be its limit function defined by \eqref{def:lambda-s}.
    \begin{enumerate}
        \item For all \( x > 0 \), we have
        \[ a_{s}(x) = \sum_{j=1}^{\infty} c(j,x) \, b^{-\alpha j}, \]
        where \( c(j,x) := s(b^j x) - b^\alpha s(b^{j-1} x) \).
        \item For all \( u, v \in \mathbb{R}^+ \) with \( u < v \), the upper box dimension of \( \Gr_{u,v}(\lambda_{\mathbf{s}}) \) satisfies
        \[ \overline{\dim}_{\mathrm{B}} \Gr_{u,v}(\lambda_{s}) \leq \max\left\{ 2 - (\alpha - \beta), 1 \right\}. \]
    \end{enumerate}
\end{lemma}

Under a slightly weaker condition, we show that $2-\alpha$ is a lower bound of Hausdorff dimension of the graph of $\lambda_\mathbf{s}$. An integer sequence $(t_n)_{n\geq 1}$ is called \emph{syndetic} if the difference sequence $\bigl(t_{n+1}-t_{n}\bigr)_{n\geq 1}$ is bounded. Precisely, we prove the following result.

\begin{thmbis}{mainthm:2}
    Let $\mathbf{s}$ be a quasi-linear sequence with $\alpha(\mathbf{s})=\alpha$ and $\beta(\mathbf{s})=0$. Assume that there is an (increasing) syndetic sequence $(t_n)_{n\geq 1}$ and a constant $c>0$ such that 
    \begin{equation}\label{eq:mainthm-2}
        a_{\mathbf{s}}(t_{n+1}b^{-k})-a_{\mathbf{s}}(t_{n}b^{-k})>c\,b^{-\alpha k}
    \end{equation} 
    for all $k\geq 1$ and for all $1\leq t_n < t_{n+1}\leq b^k-1$. Then $\dimH\Gr\lambda_\mathbf{s} = 2-\alpha$.
\end{thmbis}

\subsection{Proof of Theorem \ref{mainthm:2}}
We adopt the same strategy to prove Theorem \ref{mainthm:2}. First, we establish the relationship between the Hausdorff dimensions of the graphs of \( a_{\mathbf{s}}(x) \) and \( \lambda_{\mathbf{s}}(x) \). Then, we construct a mass distribution on the graph of \( a_{\mathbf{s}}(x) \), and use the mass distribution principle to derive a lower bound for the Hausdorff dimension of the graph of \( a_{\mathbf{s}}(x) \).

\begin{lemma}\label{lem:quasi-a-lambda}
    For all $u,v\in(0,1)$ with $u<v$, $\dimH\Gr_{u,v}(a_\mathbf{s})\leq \dimH\Gr_{u,v}(\lambda_\mathbf{s})$.
\end{lemma}
\begin{proof}
    We first show that the mapping $T_s:\,(x,\lambda_{\mathbf{s}}(x))\mapsto (x,a_{\mathbf{s}}(x))$ is Lipschitz on any closed interval $[u,v]\subset (0,1)$. For any $x_1,x_2\in[u,v]$,
    \begin{align*}
        &\ |T_s(x_1,\lambda_{\mathbf{s}}(x_1))-T_s(x_2,\lambda_{\mathbf{s}}(x_2))|^2 \\
        = &\ (x_1-x_2)^2+(a_{\mathbf{s}}(x_1)-a_{\mathbf{s}}(x_2))^2\\
         = &\  (x_1-x_2)^2+(x_1^\alpha\lambda_{\mathbf{s}}(x_1)-x_2^\alpha\lambda_{\mathbf{s}}(x_2))^2\\
        % & = (x_1-x_2)^2+(x_1^\alpha\lambda_{\mathbf{s}}(x_1)-x_2^\alpha\lambda_{\mathbf{s}}(x_1)+x_2^\alpha\lambda_{\mathbf{s}}(x_1)-x_2^\alpha\lambda_{\mathbf{s}}(x_2))^2\\
        \leq &\ (x_1-x_2)^2 + 2(x_1^\alpha-x_2^{\alpha})^2\lambda_{\mathbf{s}}^2(x_1)+2x_2^{2\alpha}(\lambda_{\mathbf{s}}(x_1)-\lambda_{\mathbf{s}}(x_2))^2.
        % \leq & \ (x_1-x_2)^2 + 2(x_1^\alpha-x_2^{\alpha})^2\max_{x\in[u,v]}\lambda_{\mathbf{s}}^2(x)+2v^{2\alpha}(\lambda_{\mathbf{s}}(x_1)-\lambda_{\mathbf{s}}(x_2))^2
    \end{align*}
    Since $\lambda_{\mathbf{s}}(\cdot)$ is continuous, $\lambda_{\mathbf{s}}(x_1)\leq \max_{x\in[u,v]}\lambda_{\mathbf{s}}(x)=:D_0$. The mean value theorem yields that there is an $x_0\in (x_1,x_2)\subset [u,v]$ such that  \[\frac{x_1^\alpha-x_2^{\alpha}}{x_1-x_2} = \alpha x_0^{\alpha - 1}\leq \alpha\max\{u^{\alpha -1}, v^{\alpha -1}\}=:D_1.\] Thus $(x_1^\alpha-x_2^{\alpha})^2\leq D_1^2(x_1-x_2)^2$. Note also $x_2^{2\alpha}\leq v^{2\alpha}$. We have 
        \begin{align*}
        &\ |T_s(x_1,\lambda_{\mathbf{s}}(x_1))-T_s(x_2,\lambda_{\mathbf{s}}(x_2))|^2 \\
        \leq &\ (x_1-x_2)^2 + 2(x_1^\alpha-x_2^{\alpha})^2\lambda_{\mathbf{s}}^2(x_1)+2x_2^{2\alpha}(\lambda_{\mathbf{s}}(x_1)-\lambda_{\mathbf{s}}(x_2))^2\\
        \leq & \ (1+ 2D_0^2D_1^2)(x_1-x_2)^2+2v^{2\alpha}(\lambda_{\mathbf{s}}(x_1)-\lambda_{\mathbf{s}}(x_2))^2\\
        \leq & \ D_3\bigl[(x_1-x_2)^2+(\lambda_{\mathbf{s}}(x_1)-\lambda_{\mathbf{s}}(x_2))^2\bigr]
    \end{align*}
    where $D_3 = \max\{1+ 2D_0^2D_1^2, 2v^{2\alpha}\}$. This shows that $T_s$ is Lipschitz. Since $\Gr_{u,v}(a_{\mathbf{s}})=T_s(\Gr_{u,v}(\lambda_{\mathbf{s}}))$, the result follows.
\end{proof} 

For all $n\geq 1$, let $\{g_n(x)\}$ be a sequence of simple functions defined as \[g_n(x)=\sum_{j=1}^{n}c(j,x)\,b^{-\alpha j}.\]
Note that for all $x\in I_{n,\,k}=[\frac{k}{b^{n}},\frac{k+1}{b^n})$, $\lfloor b^{j}x\rfloor = \lfloor b^{j-n}k\rfloor$ for $j=1,2,\dots,n$. Thus for all $x\in I_{n,\,k}$ and $j=1,2,\dots,n$, \[c(j,x)=s(b^jx)-b^{\alpha}s(b^{j-1}x)=s(b^{j-n}k)-b^{\alpha}s(b^{j-n-1}k)=c(j,k/b^{n}).\] This shows that $g_n(x)$ takes the constant value $g_n(\tfrac{k}{b^{n}})$ on each $b$-adic interval $I_{n,\,k}$. Therefore, $g_n(x)$ is a step function and its graph is composed of $b^n$ horizontal line segments $I_{n,k}\times \{g_n(\tfrac{k}{b^{n}})\}$. Let $E_{n,k}$ be a rectangle with $I_{n,k}\times \{g_n(\tfrac{k}{b^{n}})\}$ as the mid-line defined as follows 
\[
E_{n,\,k}:=I_{n,\,k}\times [g_n(\tfrac{k}{b^{n}})-D_4 C b^{-n(\alpha-\beta)},\,g_n(\tfrac{k}{b^{n}})+D_4 C b^{-n(\alpha-\beta)}]
\]
where 
\begin{equation}\label{eq:d4}
    D_4 := \Bigl\lfloor\frac{3}{b^{\alpha-\beta}-1}\Bigr\rfloor+1.
\end{equation}
Define 
\begin{equation}\label{eq:set-En}
    \mathcal{E}_n=\{E_{n,k}\}_{k=0}^{b^n-1}\quad\text{and}\quad E_n=\bigcup_{k=0}^{b^n-1}E_{n,\,k}.
\end{equation}
From the construction of $F_n$, we see that $\Gr(g_n)\subset E_n$. The upper bound of variation of $g_n(x)$ on $b$-adic intervals is given below.

\begin{lemma}\label{lem:15}
    For all $n\geq 1$, $x\in [0,1]$ and $y\in I_{n-1}(x)$, 
    \[
        |g_n(x)-g_n(y)|\leq 2C b^{-n(\alpha-\beta)}\quad \text{and}\quad |g_n(x)-g_{n-1}(y)|\leq 3C b^{-n(\alpha-\beta)},
    \]
    where the constant $C$ is given in Equation \eqref{eq:ql-2}
\end{lemma}
\begin{proof}
    By the definition of the quasi-linear sequence, there is a constant $C>0$ such that for all $x\in [0,1]$ and $n\geq 1$, 
    \[
        |c(n,x)| = |s(b^n x)-b^{\alpha}s(b^{n-1}x)|\leq C (b^{n-1}x)^{\beta}\leq C b^{n\beta}.
    \] 
    For $y\in I_{n-1}(x)$, we see that $c(j,x)=c(j,y)$ for $j=1,2,\dots,n-1$. Then 
    \begin{align*}
        |g_n(x)-g_n(y)| & = \Bigl|\sum_{j=1}^{n}c(j,x)\,b^{-\alpha j}-\sum_{j=1}^{n}c(j,y)\,b^{-\alpha j}\Bigr|\\
        & = |c(n,x)b^{-n\alpha}-c(n,y)b^{-n\alpha}| \\
        & \leq 2 C b^{-n(\alpha-\beta)}
    \end{align*}
    and 
    \begin{align*}
        |g_n(x)-g_{n-1}(y)| & \leq |g_n(x)-g_n(y)| + |g_n(y)-g_{n-1}(y)|\\
        & \leq 2 C b^{-n(\alpha-\beta)} + |c(n,y)|\cdot b^{-n\alpha}\\
        & \leq 3 C b^{-n(\alpha-\beta)}.
    \end{align*}
    The results follow.
\end{proof}
Now we show $\{\mathcal{E}_n\}$ has a nested structure.
\begin{lemma}\label{lem:nested-2}
    Let $\{\mathcal{E}_n\}_{n\geq 1}$ be families of sets given in Equation \eqref{eq:set-En}. Then for all $n\geq 1$, \[E_{n+1}\subset E_n \quad \text{and} \quad E_{n+1,\,bk+i}\subset E_{n,\,k}\] for $k=0,1,\dots,b^n-1$ and $i=0,1,\dots,b-1$.
\end{lemma}
\begin{proof}
    Let $x,y\in I_{n,\,k}$. The distance between the top edges of the rectangles $E_{n,\,k}$ and $E_{n+1,\,bk+i}$ satisfies
    \begin{align*}
        & \ \bigl(g_n(x)+D_4Cb^{-n(\alpha-\beta)}\bigr) - \bigl(g_{n+1}(y)+D_4Cb^{-(n+1)(\alpha-\beta)}\bigr) \\
        = & \ \bigl(g_n(x)-g_{n+1}(y)\bigr) + \bigl(D_4Cb^{-n(\alpha-\beta)}-D_4Cb^{-(n+1)(\alpha-\beta)}\bigr)\\
        \geq & \ -3Cb^{-(n+1)(\alpha-\beta)} + D_4C(b^{\alpha-\beta}-1)b^{-(n+1)(\alpha-\beta)}\quad \text{(by Lemma \ref{lem:15})}\\
        = & \ b^{-(n+1)(\alpha-\beta)}C\bigl(D_4(b^{\alpha-\beta}-1)-3\bigr) \geq 0
    \end{align*}
    where the last inequality follows directly from the definition of $D_4$ in \eqref{eq:d4}.  
    
    Similarly, the distance between the bottom edges of the rectangles $E_{n,\,k}$ and $E_{n+1,\,bk+i}$ satisfies
    \begin{align*}
        & \ \bigl(g_n(x)-D_4Cb^{-n(\alpha-\beta)}\bigr) - \bigl(g_{n+1}(y)-D_4Cb^{-(n+1)(\alpha-\beta)}\bigr) \\
        = & \ \bigl(g_n(x)-g_{n+1}(y)\bigr) + \bigl(D_4Cb^{-(n+1)(\alpha-\beta)}-D_4Cb^{-n(\alpha-\beta)}\bigr)\\
        \leq & \ 3Cb^{-(n+1)(\alpha-\beta)} + D_4C(1-b^{\alpha-\beta})b^{-(n+1)(\alpha-\beta)}\quad \text{(by Lemma \ref{lem:15})}\\
        = & \ b^{-(n+1)(\alpha-\beta)}C\bigl(3-D_4(b^{\alpha-\beta}-1)\bigr) \leq 0. \qquad \text{(by Definition \eqref{eq:d4})}
    \end{align*}

    Note also that $I_{n+1\, bk+i}\subset I_{n,\, k}$. We have $E_{n+1,\,bk+i}\subset E_{n,\,k}$ for all $i=0,1,\dots,b-1$ and 
    \[
        E_{n+1} = \bigcup_{k=0}^{b^{n+1}-1}E_{n+1,\,k} = \bigcup_{k=0}^{b^{n}-1}\bigcup_{i=0}^{b-1}E_{n+1,\,bk+i} \subset \bigcup_{k=0}^{b^{n}-1}E_{n,\,k} = E_n. \qedhere
    \]
\end{proof}

Now we are ready to prove Theorem \ref{mainthm:2}.

\begin{proof}[Proof of Theorem \ref{mainthm:2}]
    According to Lemma \ref{lem:quasi-a-lambda}, we only need to show $\dimH\Gr_{u,v}(a_{\mathbf{s}})\geq 2-\alpha$ for all $b$-adic intervals $I_{n,\,k}=[u,v]$. Then $\dimH\Gr_{u,v}(\lambda_{\mathbf{s}})\geq 2-\alpha$. Combining with Lemma \ref{thm:cite-quasi-linear}, in the case $\beta(s)=0$, we see that the Hausdorff dimension of the graph of $\lambda_{\mathbf{s}}$ over any sub-interval of $[0,1]$ equals $2-\alpha$.
    
    Note that $\mathcal{E}_{n}=\{E_{n,0}, E_{n,1}, \dots, E_{n,b^n-1}\}$ is a collection of disjoint sets. We define a mass distribution $\mu$ on $E_1$ by assigning $\mu(E_{n,k})=b^{-n}$ for all $n\geq 1$ and $k=0,1,\dots,b^{n}-1$. According to Lemma \ref{lem:nested-2}, we see that \[1=\mu(E_n)=\sum_{k=0}^{4^n-1}\mu(E_{n,k})~ \text{ and }~ \mu(E_{n,k})=\sum_{i=0}^{b-1}\mu(E_{n+1,bk+i}).\] For each $n\geq 1$, we define $\mu(A)=0$ for all Borel sets $A$ with $A\cap E_n=\varnothing$. Then $\mu$ can be extended to a Borel measure on $\mathbb{R}^2$ and the support of $\mu$ is contained in $\cap_{n=1}^{\infty}\bar{E}_n$. 

    For $x\in I_{n,\,k}=[u,v]$, we have 
    \begin{align*}
        |a_{\mathbf{s}}(x)-g_n(x)| & = \Bigl|\sum_{j=1}^{\infty}c(j,x)b^{-j\alpha}-\sum_{j=1}^{n}c(j,x)b^{-j\alpha}\Bigr|\\
        & = \Bigl|\sum_{j=n+1}^{\infty}c(j,x)b^{-j\alpha}\Bigr|\\
        & \leq Cb^{n\beta}\sum_{j=n+1}^{\infty}b^{-j\alpha} = \frac{C}{b^{\alpha}-1}b^{-n(\alpha-\beta)} \\ 
        & \leq D_4Cb^{-n(\alpha-\beta)}.
    \end{align*}
    Therefore, it follows from the definition of $E_{n,k}$ that $\Gr_{u,v}(a_{\mathbf{s}})\subset E_{n,\,k}$ and $\Gr(a_{\mathbf{s}})\subset E_n$ for all $n$. Thus, $\Gr(a_{\mathbf{s}})\subset\cap_{n=1}^{\infty}E_{n}$. 
    
    Let $\tilde{\mu}$ be the restriction of $\mu$ on $E_{n,k}$. Namely, for all Borel sets $A$, \[\tilde{\mu}(A):=\mu(A\cap E_{n,k})/\mu(E_{n,k})=b^{n}\mu(A\cap E_{n,k}).\] For any Borel set $U$ with $|U|<1$, we have $b^{-m-1}\leq |U| < b^{-m}$ for some integer $m\geq n$. Then $U$ is covered by a square, denoted by $S$, of side length $b^{-m}$. It follows from our assumption \eqref{eq:mainthm-2} that for all $m\geq 1$, \[a_{\mathbf{s}}(t_{n+1}b^{-m})-a_{\mathbf{s}}(t_n b^{-m})>cb^{-k\alpha}\] holds for all $1\leq t_n<t_{n+1}\leq b^{m}-1$. This implies that the distance of the top edges of $E_{k,t_{n+1}}$ and $E_{m,t_n}$ is large than $cb^{-m\alpha}$. Then the square $S$ intersects at most $\frac{b^{-m}}{cb^{-m\alpha}}$ rectangles $E_{m,t_n}$. Since the sequence $(t_n)$ is syndetic, there exists $M>0$ such that $t_{n+1}-t_n\leq M$ for all $n\geq 1$. Thus, there are at most $M$ rectangles $E_{m,*}$ between $E_{m,t_n}$ and $E_{m,t_{n+1}}$ that intersects $S$. So the square $S$ intersects at most $\frac{b^{-m}}{cb^{-m\alpha}}\cdot M$ rectangles $E_{m,*}$. Recall that when $E_{m,*}\subset E_{n,k}$, we have $\tilde{\mu}(E_{k,*})=b^{n-m}$. Consequently, 
    \begin{align*}
        \tilde{\mu}(U) \leq \tilde{\mu}(S) & \leq \frac{b^{-m}}{cb^{-m\alpha}}\cdot M \cdot b^{n-m}\\
        &  = \frac{b^n M}{c}\cdot b^{-(2-\alpha)m}\\
        & \leq \frac{b^nb^{2-\alpha} M}{c}\cdot |U|^{2-\alpha}.
    \end{align*} 
    By Lemma \ref{lem:mdp} (the Mass Distribution Principle), we have $\dimH\Gr_{u,v}(a_{\mathbf{s}}) \geq 2-\alpha$.
\end{proof}

\subsection{Examples}
Although our goal is to generalize Theorem \ref{mainthm:1} to encompass more quasi-linear sequences, not all quasi-sequences satisfy the conditions of Theorem \ref{mainthm:2}: (1) $\alpha>0$ and $\beta=0$; (2) Condition \eqref{eq:mainthm-2}. To elaborate, we provide three illustrative examples: Example \ref{eg:16} satisfies both (1) and (2), whereas Example \ref{eg:18} fails to satisfy (1), and Example \ref{eg:19} satisfies (1) but fails to satisfy (2).

\begin{example}[Sum of the Thue-Morse sequence]\label{eg:16}
    Let $\mathbf{m}=(m(n))_{n\geq 0}$ be the Thue-Morse sequence defined by $m(0)=0$, $m(2i)=m(i)$ and $m(2i+1)=1-m(i)$ for all $i\geq 0$. Denote by $\mathbf{s}=(s(n))_{n\geq 1}$ the sum of $\mathbf{m}$ where
    \begin{equation*}
        s(n)=\sum_{i=0}^{n-1}m(i).
    \end{equation*}
    Then $\alpha(\mathbf{s})=1$ and $\beta(\mathbf{s})=0$. The sequence $(s(n))_{n\geq 1}$ is quasi-linear (in base $b=2$). Next we show that the condition \eqref{eq:mainthm-2} holds for $(t_n=2n)_{n\geq 1}$. For any $k\geq 1$ and $t_n,\,t_{n+1}< 2^k$,
    \begin{align*}
        a_{\mathbf{s}}(t_{n+1}2^{-k}) - a_{\mathbf{s}}(t_n2^{-k}) & = \bigl(t_{n+1}2^{-k}\cdot\lambda_{\mathbf{s}}(t_{n+1}2^{-k}) - s(0) \bigr) - \bigl(t_n2^{-k}\cdot\lambda_{\mathbf{s}}(t_n2^{-k}) - s(0) \bigr)\\
        & = t_{n+1}2^{-k}\cdot\lambda_{\mathbf{s}}(t_{n+1}2^{-k}) - t_n2^{-k}\cdot\lambda_{\mathbf{s}}(t_n2^{-k})\\
        & = \lim_{j\to\infty}\frac{s(t_{n+1}2^{j-k})}{2^{j-k}} - \lim_{j\to\infty}\frac{ s(t_n2^{j-k})}{2^{j-k}}\\
        & =\lim_{j\to\infty}\frac{s(t_{n+1}2^{j-k})-s(t_n2^{j-k})}{2^{j}}.
    \end{align*}
    Note that for all $j\geq 1$, $m(2j)+m(2j+1)=m(j)+1-m(j)=1$. Thus for all $j> k$, 
    \begin{align*}
        s(t_{n+1}2^{j-k})-s(t_n2^{j-k}) & = m(t_n2^{j-k})+\cdots+m(t_{n+1}2^{j-k}-1)\\
        & = \sum_{\ell=0}^{2^{j-k}-1}\bigl[m(t_n2^{j-k}+2\ell)+m(t_n2^{j-k}+2\ell+1)\bigr] \\
        & = \sum_{\ell=0}^{2^{j-k}-1} 1 = 2^{j-k}.
    \end{align*}
    Thus, for all \( k \geq 1 \), we have  
    \[
    a_{\mathbf{s}}\left(t_{n+1}2^{-k}\right) - a_{\mathbf{s}}\left(t_n2^{-k}\right) = 2^{-k},
    \]  
    and condition \eqref{eq:mainthm-2} is therefore satisfied. By Theorem \ref{mainthm:2}, it follows that the Hausdorff dimension of the graph of the limit function of the sum of the Thue-Morse sequence is \( 1 \).
\end{example}
\begin{example}[Double sum of the Thue-Morse sequenc]\label{eg:18}
    Denoted by $\mathbf{s}'$ the sum of $\mathbf{s}$ given in Example \ref{eg:16}. It is also a quasi-linear sequence (in base $2$). However, $\alpha(\mathbf{s}')=2$ and $\beta(\mathbf{s}')=1\neq 0$. In this case, Theorem \ref{mainthm:2} can not be applied.
\end{example}

\begin{example}[Condition \eqref{eq:mainthm-2} might fail to be satisfied]\label{eg:19}
    Let $s(n)=\sum_{i=0}^{n-1}r(i)$ be the sum of the Rudin-Shapiro sequence $\mathbf{r}$ investigated in \cite{BM1978,BEM1983}. Then $\mathbf{s}=(s(n))_{n\geq 1}$ is quasi-linear (in base $b=4$) with $\alpha(\mathbf{s})=1/2$ and $\beta(\mathbf{s})=0$. 
    
    Let $(t_n)_{n\ge 1}$ be an increasing syndetic sequence. Namely, there exists constants $L>0$ such that $t_{n+1}-t_n\leq L$ for all $n\geq 1$. Let $k\geq 1$ and $n\geq 1$ with $1\leq t_n<t_{n+1}<4^{k}$. 
    % If \[a_{\mathbf{s}}(t_{n+1}4^{-k})-a_{\mathbf{s}}(t_{n}4^{-k})>c\,2^{-k},\] then 
    A calculation similar as in Example \ref{eg:16} yields that 
    \begin{align}
        a_{\mathbf{s}}(t_{n+1}4^{-k}) - a_{\mathbf{s}}(t_n4^{-k}) & = \bigl(t_{n+1}2^{-k}\cdot\lambda_{\mathbf{s}}(t_{n+1}2^{-k}) - s(0) \bigr) - \bigl(t_n2^{-k}\cdot\lambda_{\mathbf{s}}(t_n2^{-k}) - s(0) \bigr)\notag\\
        & = t_{n+1}2^{-k}\cdot\lambda_{\mathbf{s}}(t_{n+1}2^{-k}) - t_n2^{-k}\cdot\lambda_{\mathbf{s}}(t_n2^{-k})\notag\\
        & =  t_{n+1}4^{-k}\lim_{j\to\infty}\frac{s(t_{n+1}4^{j-k})}{\sqrt{t_{n+1}4^{j-k}}} - t_n 4^{-k}\lim_{j\to\infty}\frac{ s(t_n4^{j-k})}{\sqrt{t_n4^{j-k}}}\notag\\
        % & =\sqrt{t_{n+1}}2^{-k}\lim_{j\to\infty}\frac{s(t_{n+1}4^{j-k})}{2^{j}} - \sqrt{t_n}2^{-k}\lim_{j\to\infty}\frac{ s(t_n4^{j-k})}{2^{j}}\\
        & \leq 2^{-k}\lim_{j\to\infty}\frac{s(t_{n+1}4^{j-k})-s(t_n4^{j-k})}{2^{j}}.\label{eq:eg2-1}
    \end{align}
    For all large enough $j$, it follows from the recurrence relations \eqref{def:rs} that 
    \begin{align}
        s(t_{n+1}4^{j-k})-s(t_n4^{j-k}) & \leq s\bigl((t_{n}+L)4^{j-k}\bigr)-s(t_n4^{j-k})\notag\\
        & = r(t_n4^{j-k})+r(t_n4^{j-k}+1)+\cdots + r\bigl((t_{n}+L)4^{j-k}-1\bigr)\notag\\
        & = \sum_{i=0}^{L4^{j-k-1}-1}\sum_{h=0}^{3}r(t_n4^{j-k}+4i+h)\notag\\
         & = 2\sum_{i=0}^{L4^{j-k-1}-1}r(t_n4^{j-k-1}+i)\qquad \text{by \eqref{def:rs}}\notag\\
         & = 2\bigl(s((t_{n}+L)4^{j-k-1})-s(t_n4^{j-k-1})\bigr)\notag\\
         & = \cdots = 2^{j-k}(s(t_n+L)-s(t_n))\notag\\
         & \leq 2^{j-k}L.\label{eq:eg2-2}
    \end{align}
    Combining Equations \eqref{eq:eg2-1} and \eqref{eq:eg2-2}, we have \[a_{\mathbf{s}}(t_{n+1}4^{-k}) - a_{\mathbf{s}}(t_n4^{-k})\leq 4^{-k}L.\]
    If Condition \eqref{eq:mainthm-2} holds, then there exists a constant $c>0$ such that 
    \[4^{-k}L\geq a_{\mathbf{s}}(t_{n+1}4^{-k}) - a_{\mathbf{s}}(t_n4^{-k})\geq c\, 2^{-k}.\]
    However, for all sufficiently large $k$, the upper bound $4^{-k}L$ is smaller than the lower bound $c\,4^{-k/2}$, which is a contradiction.
\end{example}

\section*{Acknowledgement}
The first author is supported by Guangzhou Basic and Applied Basic Research Projects (Outstanding Doctoral Researchers, ``Sustained Development'', Grant No. 2025A04J5320) and the National Natural Science Foundation of China (Grant Nos. 12371086, 12271175).

\end{document}